\documentclass[oribibl]{llncs}

\usepackage{amsmath,amssymb,amsfonts, enumerate}
\usepackage{xcolor, tikz, ifthen, calc}

\newcommand{\Z}{\mathbb{Z}}
\newcommand{\R}{\mathbb{R}}
\newcommand{\floor}[1]{\lfloor#1\rfloor}

\newcommand{\supp}{\operatorname{supp}}
\newcommand{\cone}{\operatorname{cone}}

\newcommand{\asy}{\text{asy}}

\newenvironment{cpl}{\begin{trivlist} \item[] {\em Proof of Lemma.}}{\hspace*{\stretch{1}} \qed \end{trivlist}}

\spnewtheorem*{exam}{Example}{\bfseries}{\itshape}

\title{Sparsity of integer solutions in the average case}
\author{Timm Oertel\inst{1}
\and Joseph Paat\inst{2}
\and Robert Weismantel\inst{2}}
\institute{
School of Mathematics, Cardiff University, United Kingdom\and
Institute for Operations Research, ETH Z\"urich, Switzerland}

\begin{document}
\maketitle
\begin{abstract}
We examine how sparse feasible solutions of integer programs are, on average.
Average case here means that we fix the constraint matrix and vary the right-hand side vectors.
For a problem in standard form with $m$ equations, there exist LP feasible solutions with at most $m$ many nonzero entries. 
We show that under relatively mild assumptions, integer programs in standard form have feasible solutions with $O(m)$ many nonzero entries, on average.
Our proof uses ideas from the theory of groups, lattices, and Ehrhart polynomials.
From our main theorem we obtain the best known upper bounds on the integer Carath\'eodory number provided that the determinants in the data are small.
\end{abstract}

\section{Introduction}

Let $m,n \in \Z_{\ge 1}$ and $A \in \Z^{m \times n}$.
We always assume that $A$ has full row rank.
We also view $A$ as a set of its column vectors.
So, $W \subseteq A$ implies that $W$ is a subset of the columns of $A$.

We aim to find a sparse integer vector in the set
\[
P(A,b) := \{x\in \Z^n_{\ge 0}: Ax = b\},
\]
where $b \in \Z^m$. 
That is, we aim at finding a solution $z$ such that $|\supp(z)|$ is as small as possible, where $\supp(x) := \{i \in \{1, \dotsc, m\} : x_i \neq 0\}$ for $x \in \R^n$. 
To this end, we define the \emph{support function of $(A,b)$} to be
\[
\sigma(A,b) := \min\{|\supp(z)| : z \in  P(A,b)\}.
\]
If $P(A,b) = \emptyset$, then $\sigma(A,b) := \infty$.
We define the \emph{support function of $A$} to be
\[
\sigma(A) := \max\{\sigma(A,b) :  b \in \Z^m \text{ and } \sigma(A,b) < \infty\}. 
\]

The question of determining $\sigma(A)$ generalizes problems that have been open for decades.
A notable special case is the so-called integer Carath\'eodory number, i.e. the minimum number of Hilbert basis elements in a rational pointed polyhedral cone required to represent an integer point in the cone.
We say that $A$ has the \emph{Hilbert basis property} if its columns correspond to a Hilbert basis of $\cone(A)$.
For $A$ with the Hilbert basis property, Cook et al.~\cite{CookFS1986} showed that $\sigma(A) \le 2m-1$ and Seb\H{o} showed that $\sigma(A) \le 2m - 2 $~\cite{Sebo1990}.
Bruns et al.~\cite{BrunsGHMW1999} provide an example of $A$ with the Hilbert basis property with $ \frac{7}{6} m \le \sigma(A) $.
However, for matrices with the Hilbert basis property, the true value of $\sigma(A)$ is unknown.

For general choices of $A$, Eisenbrand and Shmonin~\cite{ES2006} showed that $\sigma(A) \le 2m \log_2(4m\|A\|_{\infty})$, where $\|\cdot\|_{\infty}$ is the max norm.
Aliev et al.~\cite{ADOO2017} and Aliev et al.~\cite{ADEOW2018} improved the previous result and showed that
\begin{equation}\label{eqDetBound}
\textstyle
\sigma(A) \le m + \log_{2}(g^{-1}\sqrt{\det(A{\displaystyle A^{\intercal}})}) \le 2m\log_2(2\sqrt{m}\|A\|_{\infty}),
\end{equation}
where $g = \gcd\{|\det(B)|: B \text{ is an } m \times m \text{ submatrix of } A\}$.
It turns out that the previous upper bound is close to the true value of $\sigma(A)$.
In fact, for every $\epsilon > 0$, Aliev et al.~\cite{ADEOW2018} provide an example of $A$ for which $m \log_2(\|A\|_{\infty})^{1/(1+\epsilon)} \le \sigma(A)$. 

In this paper, we consider $\sigma(A,b)$ for \emph{most choices} of $b$. 
We formalize this `average case' using the \emph{asymptotic support function of $A$} defined by
\[
\sigma^{\asy}(A) := \min\bigg\{ k \in \Z : \lim_{t \to \infty} \frac
{|\{b \in \{-t, ..., t\}^m : \sigma(A,b) \le k\}|}
{|\{b \in \{-t, ..., t\}^m : P(A,b) \neq \emptyset\}|}
= 1\bigg\}.
\]
Note that $\sigma^{\asy}(A) \le \sigma(A) \le |A|$. 

The value $\sigma^{\asy}(A)$ can be thought of as the smallest $k$ such that almost all feasible integer programs with constraint matrix $A$ have solutions with support of cardinality at most $k$.
The function $\sigma^{\asy} ( \cdot )$ was introduced by Bruns and Gubeladze in~\cite{BG2004}, where it was shown that $\sigma^{\asy}(A) \le 2m-3$ for matrices with the Hilbert basis property.
In general, an average case analysis of the support question has not been provided in the literature. 
Average case behavior of integer programs has been studied in specialized settings, see, e.g.,~\cite{DF1989} for packing problems in $0,1$ variables and~\cite{IHO2017} for problems with only one constraint.
However, to the best of our knowledge, there are no other studies available that are concerned with the average case behavior of integer programs, in general.

Our analysis reveals that the \emph{sizes} of the $m \times m$ minors of $A$ affect sparsity. 
It turns out that \emph{the number of factors in the prime decomposition} of the minors also affects sparsity.
Moreover, for matrices with \emph{large} minors but \emph{few} factors, there exist solutions whose support depends on the number of factors rather than the size of the minors.
Recall that a prime is a natural number greater than or equal to $2$ that is divisible only by itself and $1$.
We now formalize these parameters related to the minors of a matrix.

Let $W \in \Z^{m \times d}$ be of full row rank, where $d \in \Z_{\ge 1}$.
Denote the set of absolute values of the $m \times m$ minors by
\[
\Delta(W) := \{|\det(W')|: W' \text{ is an invertible } m \times m \text{ submatrix of } W\},
\]
and denote the set of `number of prime factors' in each minor by
\begin{equation}\label{eqPhiDefn}
\Phi(W) :=  \bigg\{t \in \Z_{\ge 1}: 
\begin{array}{l}
W' \text{ an invertible } m \times m \text{ submatrix of } W, \text { and}\\
|\det(W')| = \prod_{i=1}^t \alpha_i \text{ with } \alpha_1, \dotsc, \alpha_t \text{ prime}
\end{array}
\bigg\}.
\end{equation}
If $\Phi(W)$ consists of only one element (e.g., when $W \in \Z^{m\times m}$), then we denote the element by $\phi(W)$.
If $W \in \Z^{m\times m}$ and $|\det(W)| = 1$, then $\phi(W) = 0$.
We denote the maximum and minimum of these sets by
\[
\begin{array}{rclcrcl}
\delta^{\max}(W) & := & \max ( \Delta(W)), & &\delta^{\min}(W) & := & \min ( \Delta(W)), \\[.1 cm]
\phi^{\max}(W) & := & \max ( \Phi(W)), & \quad \text{and} \quad & \phi^{\min}(W) & := & \min ( \Phi(W)).
\end{array}
\]
Our first main result bounds $\sigma^{\asy}$ using these parameters.
\begin{theorem}\label{thmMainProb}
Let $A \in \Z^{m\times n}$ and $W \subseteq A$ such that 
\(
\cone(A) = \cone(W).
\)
Then

\begin{enumerate}[(i)]
\item $\sigma^{\asy}(A) \le m + \phi^{\max}(W) \le m + \log_{2}\left(\delta^{\max}(W)\right)$,\\
\item $\sigma^{\asy}(A) \le 2m + \phi^{\min}(W) \le 2m + \log_{2}\left(\delta^{\min}(W)\right)$.
\end{enumerate}
\end{theorem}

Theorem~\ref{thmMainProb} guarantees that the average support $\sigma^{\asy}(A)$ is linear in $m$ in two special cases: (a) the minimum minor of $A$ is on the order of $2^m$ or (b) there is a prime minor.
We emphasize that \emph{(ii)} uses the \emph{minimum} values $\phi^{\min}$ and $\delta^{\min}$, which can be bounded by sampling any $m\times m$ invertible submatrix of $A$.
Thus, $\sigma^{\asy}(A)$ can be bounded by finding a single $m\times m$ invertible submatrix of $A$.

Note that the bound in~\eqref{eqDetBound} includes the term $g$.
Our proof of Theorem~\ref{thmMainProb} can be adjusted to prove $\sigma^{\asy}(A) \le m + \log_2(g^{-1}\delta^{\max}(W))$ and $\sigma^{\asy}(A) \le 2m + \log_2(g^{-1}\delta^{\min}(W))$. 
We omit this analysis here to simplify the exposition.
However, it should be mentioned that
\[
\textstyle
\delta^{\max}(A) \le (\sum_{\delta \in \Delta(A)} \delta^2)^{1/2} =  \sqrt{\det(A{\displaystyle A^{\intercal}})},
\]
where the equation follows from the so-called Cauchy-Binet formula.
Therefore, if $A$ has two nonzero $m\times m$ minors, then Theorem~\ref{thmMainProb} \emph{(i)} improves~\eqref{eqDetBound}, on average.

A corollary of Theorem~\ref{thmMainProb} is that if $A$ has the Hilbert basis property, then the extreme rays of $\cone(A)$ provide enough information to bound $\sigma^{\asy}(A)$.
\begin{corollary}\label{thmMainHilb}
Let $V \subseteq \Z^m$ and $H \subseteq \Z^{m \times t}$.
Assume that $H$ has the Hilbert basis property and $\cone(H) = \cone(V)$.
Then
\[
\sigma^{\asy}(H) \le m + \phi^{\max}(V)  \le m + \log_{2}\left(\delta^{\max}(V)\right).
\]
\end{corollary}
If $\delta^{\max}(V) < 2^{m-3}$, then the bound in Corollary~\ref{thmMainHilb} improves the bound in~\cite{BG2004}.

By modifying a construction in~\cite{ADEOW2018}, we obtain two interesting examples of $\sigma^{\asy}(A)$.
The first example shows that Theorem~\ref{thmMainProb} \emph{(i)} gives a tight bound.
The second example shows that Theorem~\ref{thmMainProb} \emph{(ii)} gives a tight bound and that $\sigma^{\asy}(A)$ can be significantly smaller than $\sigma(A)$.

\begin{theorem}\label{thmMainLower}
For every $m \in \Z_{\ge 1}$ and $d \in \Z_{\ge 1}$, there is a matrix $A \in \Z^{m \times n}$ such that $\phi^{\max}(A) = d$ and $\sigma^{asy}(A) = m+d$.

For every $m \in \Z_{\ge 1}$ and $d \in \Z_{\ge m+3}$, there is a matrix $B \in \Z^{(m+1) \times n}$ such that $\phi^{\min}(B) =0$ and 
\(
\sigma^{asy}(B) = 2m+2 < m+d = \sigma(B).
\)
\end{theorem}

The proof of Theorem~\ref{thmMainProb} is based on a combination of group theory, lattice theory, and Ehrhart theory.
On a high level, the combination of group and lattice theory bears similarities to papers of Gomory~\cite{G1965} and Aliev et al.~\cite{ADOO2017}.
Gomory investigated the value function of an IP and proved its periodicity when the right-hand side vector is sufficiently large. 
Aliev et al. showed periodicity for the function $\sigma(A, b)$ provided again that $b$ is sufficiently large. 
Our refined analysis allows us to quantify the number of right-hand sides for which the support function is small.
This new contribution requires not only group and lattice theory, but also Ehrhart theory.

Sections~\ref{secGroup} and~\ref{secCones} we provide background on groups and subcones.
In Section~\ref{secMainProb} we use the average support for each subcone to prove Theorem~\ref{thmMainProb}.
We prove Theorem~\ref{thmMainLower} in Appendix~\ref{secMainLower}.

\section{The group structure of a parallelepiped}\label{secGroup}

Let $W \in \Z^{m\times m}$ be an invertible matrix, which we also view as a set of $m$ linearly independent column vectors.
Let $\Pi(W)$ denote the integer vectors in the fundamental parallelepiped generated by $W$:
\[
\Pi(W) := \{z \in \Z^m : z = W \lambda  \text{ for } \lambda \in [0,1)^m\}.
\]

For each $b \in \Z^m$, there is a unique $g \in \Pi(W)$ such that $ b = g + Wz$, where $z \in \Z^m$~\cite[Lemma 2.1, page 286]{barv}.
Thus, we can define a \emph{residue function} $\rho_W : \Z^m \to \Pi(W)$ by
\begin{equation}\label{eqResidue}
\rho_W(b) = \rho_W (g + Wz ) \mapsto g.
\end{equation}
The image of $\Z^m$ under $\rho_W$ (i.e., $\Pi(W)$) creates a group $G_W(\Z^m) $ using the operation $+_{G_W} : \Pi(W)\times \Pi(W) \to \Pi(W)$ defined by
\[
g +_{G_W} h \mapsto \rho_W(g+h).
\]
The identity of $G_W(\Z^m)$ is the zero vector in $\Z^m$, and 
\begin{equation}\label{eqDetCardinality}
|G_W(\Z^m)| = |\det(W)|,
\end{equation}
see, e.g., ~\cite[Corollary 2.6, page 286]{barv}.
Equation~\eqref{eqDetCardinality} implies $G_W(\Z^m)$ is finite.

The choice of notation for $G_{W}(\Z^m)$ is to emphasize that it is the group generated by the residues of all integer linear combinations of vectors in $\Z^m$.
We can also consider the group generated by any subset of vectors in $\Z^m$.
Given $B \subseteq \Z^m$, we denote the subgroup of $G_W(\Z^m)$ generated by $B$ by
\begin{equation}\label{eqSubgroup}
G_W(B) := \{\rho_W(Bz) : z \in \Z^{|B|}\}.
\end{equation}
If $B = \emptyset$, then $G_W(B) := \{0\}$.
The set $G_W(B)$ is a subgroup of $G_W(\Z^m)$ because $\{Bz : z \in \Z^{|B|}\}$ is a sublattice of $\Z^m$.

We collect some basic properties about the group $G_W(B)$. 
\begin{lemma}\label{lemSubgroup}
Let $W \in \Z^{m\times m}$ be an invertible matrix.
For every $B \subseteq \Z^m$, $G_W(B) = \{\rho_W(Bz) : z \in \Z^{|B|}_{\ge 0}  \}$.
\end{lemma}
\begin{proof}
For each $z \in \Z^{|B|}$, we can write $Bz$ as
\[
\textstyle
Bz = \sum_{b \in B : z_b \ge 0} z_b b + \sum_{b \in B : z_b < 0} z_b b.
\]
Thus, it suffices to show $\rho_W(-b) \in  \{\rho_W(By) : y \in \Z^{|B|}_{\ge 0}  \}= : C$ for each $b \in B$.
If $\rho_W(b) = 0$, then $\rho_W(-b) = \rho_W(b) = 0 \in  C$.
If $\rho_W(b) \neq 0$, then because $G_W(B)$ is finite there exists $\tau \in \Z_{\ge 2}$ with $\rho_W(\tau b ) = 0$.
Note that $\rho_W((\tau-1)b) + \rho_W(b) = 0 = \rho_W(b) + \rho_W(-b)$, so $\rho_W(-b) = \rho_W((\tau-1)b) \in  C$.
\qed
\end{proof}
\begin{lemma}\label{lemGroupDecomp}
Let $W \in \Z^{m\times m}$ be an invertible matrix and $B \subseteq \Z^m$. 
If $t \in \Z_{\ge 0}$ with $t \ge \phi(W)$, then there exist $w^1, \dotsc, w^t \in B$ (possibly with repetitions) such that $G_W(\{w^1, \dotsc, w^t\}) = G_W(B)$.
\end{lemma}

\begin{proof}
Set $s:= \phi(W)$.
First, we show that for each $r \in \{0, \dotsc, s\}$ there exist $w^1, \dotsc, w^r \in B$ (possibly with repetitions) such that
\begin{equation}\label{eqInductionPrime}
\begin{array}{r@{\hskip .5 cm}l}
\text{either}& G_W(\{w^1, \dotsc, w^r\}) = G_W(B) \\[.1 cm]
\text{or} & G_W(\emptyset) \subsetneq G_W(\{w^1\}) \subsetneq \dotsc \subsetneq G_W(\{w^1, \dotsc, w^r\}).
\end{array}
\end{equation}
We prove~\eqref{eqInductionPrime} by induction on $r$.
The result is vacuously true for $r = 0$, so assume that~\eqref{eqInductionPrime} holds for $r \in \Z_{\ge 0}$ and consider $r+1$.
Define 
\begin{equation}\label{eqDefineStep}
 G^r := G_W(\{w^1, \dotsc, w^r\}).
\end{equation}

By the induction hypothesis, there exist $w^1, \dotsc, w^r \in B$ such that~\eqref{eqInductionPrime} holds. 
If $G^r = G_W(B)$, then $w^{r+1} := w^r$ proves~\eqref{eqInductionPrime} for $r+1$.
If $G^r \subsetneq G_W(B)$, then $G^0 \subsetneq \dotsc \subsetneq G^r$ by~\eqref{eqInductionPrime} and induction.
Recall $\rho_W(\cdot)$ from~\eqref{eqResidue}.
If $\rho_W(b) \in G^r$ for every $b \in B$, then $G_W(B) \subseteq G^r$ and $|G_W(B)| \le |G^r| < |G_W(B)|$, which is a contradiction. 
Thus, there exists $w^{r+1} \in B$ such that $\rho_W(w^{r+1}) \not \in G^r$. 
The sequence $G^0, \dotsc, G^r$,  
\(
G^{r+1} := G_W(\{w^1, \dotsc, w^{r+1}\})
\)
satisfies~\eqref{eqInductionPrime}, which proves~\eqref{eqInductionPrime}.

Let $G^1, \dotsc, G^s$ be chosen to satisfy~\eqref{eqInductionPrime}.
If $G^s = G_W(B)$, then set $w^{s+1} = \dotsc = w^t := w^s$ to conclude $G_W(\{w^1, \dotsc, w^t\}) = G_W(B)$.
It is left to consider the case when $G^s \subsetneq G_W(B)$.
We claim that this leads to a contradiction.

By~\eqref{eqPhiDefn} and~\eqref{eqDetCardinality}, $|G_W(\Z^m)| = \prod_{i=1}^s \alpha_i $ for primes $\alpha_1, \dotsc, \alpha_s $.
By~\eqref{eqDefineStep}, $G^1, \dotsc, G^s$ are subgroups of $G_W(\Z^m)$, so $|G^1|, \dotsc, |G^s| $ divide $|G_W(\Z^m)|$ (see, e.g.,~\cite[Chapter 2]{artin}).
Also, $G^s \subsetneq G_W(B)$ and~\eqref{eqInductionPrime} imply that $G^1 \subsetneq \dotsc \subsetneq G^s$.
Hence, $1 < |G^1| < \dotsc < |G^s|$ and $|G^i|$ divides $|G^{i+1}|$ for each $i \in \{1, \dotsc, s-1\}$. 
This implies that $|G^s|$ has at least $s$ many prime factors. 
However, $|G^s| < |G_W(B)| \le |G_W(\Z^m)|$, and $|G_W(\Z^m)|$ only has $s$ many prime factors.
Thus, 
\(
|G^{i}| = |G_W(\Z^m)|
\) 
for some $i \in \{1, \dotsc, s\}$, which contradicts $G^i = G_W(\Z^m) \supseteq G_W(B)$.
\qed
\end{proof}

\section{Lattice points in cones}\label{secCones}

%
A set $\Lambda \subseteq \Z^m$ is a lattice if $0 \in \Lambda$, $x + y \in \Lambda$ for $x,y\in \Lambda$, and if $x \in \Lambda$ then $-x \in \Lambda$ (see, e.g.,~\cite[Chapter VII]{barv}). 
So, $\Lambda$ is a subgroup of $\Z^m$.
We assume that a lattice contains $m$ linearly independent vectors.
For $B \subseteq \R^m$ and $x \in \R^m$, set $B + x := \{b + x : b \in B\}$.

We use following lemma to find suitable translated subcones in which $\sigma(A, \cdot)$ is bounded. 
The proof of Lemma~\ref{lemOverlappingConesMany} is in Appendix~\ref{appProoflemOverlappingConesMany}.
\begin{lemma}\label{lemOverlappingConesMany}
Let $v^1, \dotsc, v^m \in \Z^m$ be linearly independent vectors and set $K := \cone(\{v^1, \dotsc, v^m\})$.
For $t \in \Z_{\ge 0}$ and $x^1, \dotsc, x^t \in \Z^m$, there is a $z = \sum_{i=1}^m k_i v^i \in K \cap \Z^m$, where $k_1, \dotsc, k_m \in \Z_{\ge 0}$, such that $ K+z \subseteq K \cap \bigcap_{i=1}^t (K+x^i)$.
\end{lemma}

Let $W \subseteq \Z^m$.
For each $x \in \cone(W)$, Carath\'eodory's Theorem implies that there is a linearly independent set $W^i \subseteq W$ such that $x \in \cone(W^i)$. 
Thus,
\begin{equation}\label{eqCover}
\textstyle 
\cone(W) = \bigcup_{i=1}^s \cone(W^i),
\end{equation}
where $s \in \Z_{\ge 1}$ and $W^1, \dotsc, W^s \subseteq W$ are the linearly independent subsets of $W$.
The following lemma states that for a given lattice $\Lambda$, `most' of the points in $\cone(W) \cap \Lambda$ are found in translations of the subcones $\cone(W^1), \dotsc, \cone(W^s)$.

\begin{lemma}\label{lemAsymptoticGeneral}
Let $W \subseteq \Z^m$ be such that $\cone(W)$ is $m$-dimensional.
Let $s \in \Z_{\ge 1}$ and $W^1 ,\dotsc, W^s \subseteq W$ be as in~\eqref{eqCover}. 
Let $\Lambda \subseteq \Z^m$ be a lattice and assume that $W^1, \dotsc, W^s \subseteq \Lambda$.
For each $i \in \{1, \dotsc, s\}$, choose any $k_w \in \Z_{\ge 0}$ for each $w \in W^i$, and define $z^i  := \sum_{w \in W^i} k_w w$. 
Then
\begin{equation}\label{eqLimit1}
\lim_{t \to \infty} 
\frac
{|\{-t, ..., t\}^m \cap \bigcup_{i=1}^s (\Lambda \cap (\cone(W^i)+z^i))|}
{|\{-t, ..., t\}^m \cap \bigcup_{i=1}^s (\Lambda \cap \cone(W^i))|} =  1.
\end{equation}
\end{lemma}
\begin{proof}
For $i \in \{1, \dotsc, s\}$ set $K^i := \cone(W^i)$.
The fraction in~\eqref{eqLimit1} equals 
\[
 1 - \frac{|\{-t, ..., t\}^m \cap \bigcap_{i=1}^s (\Lambda \cap [\cone(W) \setminus (K^i+z^i)])|}
{|\{-t, ..., t\}^m \cap  \bigcup_{i=1}^s (\Lambda \cap K^i)|},
\]
which is at least as large as
\[
 1 - \frac{|\{-t, ..., t\}^m \cap \bigcup_{i=1}^s (\Lambda \cap [K^i \setminus (K^i+z^i)])|}
{|\{-t, ..., t\}^m \cap  \bigcup_{i=1}^s (\Lambda \cap K^i)|}.
\]
Thus, in order to prove~\eqref{eqLimit1}, it is enough to prove 
\begin{equation}\label{eqLimit0}
\lim_{t \to \infty}
\frac{|\{-t, ..., t\}^m \cap \bigcup_{i=1}^s (\Lambda \cap [K^i \setminus (K^i+z^i)])|}
{|\{-t, ..., t\}^m \cap  \bigcup_{i=1}^s (\Lambda \cap K^i)|} = 0.
\end{equation}
By assumption, $\cone(W)$ is $m$-dimensional.
Thus, we may assume that the sets $W^1, \dotsc, W^s$ each have $m$ linearly independent vectors.

Let $i \in \{1, \dotsc, s\}$ and $L^i \subseteq \Lambda \cap K^i$ be the $\Lambda$ points that are coordinate-wise at most one more than $z^i$ in the coordinate system defined by $W^i$:
\[
\textstyle
L^i := \{ \sum_{w \in W^i} \beta_w w : \beta_w \in \R \text{ and } 0 \le \beta_w \le k_w+1~\forall~ w \in W^i \} \cap \Lambda.
\]
The set $L^i$ is finite.

The numerator of~\eqref{eqLimit0} considers $\Lambda \cap [K^i \setminus (K^i + z^i)]$, so take $y \in \Lambda \cap [K^i \setminus (K^i+z^i)]$.
We claim that
\begin{equation}\label{eqYinHPlane}
\textstyle 
y \in r + \{ \sum_{w \in I} \lambda_w w : \lambda_w \in \R_{\ge 0} ~\forall ~ w \in I \},
\end{equation}
where $r \in L^i$ and $I \subseteq W^i$ with $|I| \le m-1$.
Write $y$ as 
\(
y = \sum_{w \in W^i} \gamma_w w,
\)
where $\gamma_w \in \R_{\ge 0}$ for each $w \in W^i$ and $\gamma_{\bar{w}} < k_{\bar{w}}$ for some $\bar{w} \in W^i$.
We have $y - \tau w \in \Lambda$ for each $w \in W^i\setminus\{\bar{w}\}$ and $\tau \in \Z$ because $W^i\subseteq \Lambda$ and $y \in \Lambda$.
In particular, $y -\floor{\gamma_w} w \in \Lambda \cap K^i$ and
\(
y - \sum_{w \in V} \floor{\gamma_w} w \in L^i,
\)
where $V := \{w \in W^i : \gamma_w > k_w +1\}$.
This proves~\eqref{eqYinHPlane}.
Note that we use the fact that $L^i$ is defined by $\beta_w \le k_w+1$ rather than $\beta_w \le k_w$:
if $L^i$ was defined by $\beta_w \le k_w$, then in the extreme case $0 = k_w$ and $\gamma_w \in (0,1)$, the vector $y - \floor{\gamma_w} w = y$ is not in $L^i$.

We use the fact that $|I| < m$ to show $\Lambda \cap [K^i \setminus (K^i + z^i)]$ is contained in finite union of lower dimensional spaces. 
Although we showed $|I| \le m-1$, we can assume $|I| = m-1$ by extending it arbitrarily to have $m-1$ columns and setting $\lambda_w = 0$ for these new columns.
Hence,
\begin{align}\label{eqBadSet}
&   \bigcup_{i=1}^s \Lambda \cap [K^i \setminus (K^i+z^i)]\nonumber\\
\subseteq &
\bigcup_{i = 1}^s ~
\bigcup_{r \in L^i} ~
\bigcup_{\substack {I \subseteq W^i\\ |I| = m-1}} r + \bigg\{ \sum_{w \in I} \lambda_w w : \lambda_w \in \R_{\ge 0} ~\forall ~ w \in I \bigg\}.
\end{align}
For each $i \in \{1, \dotsc, s\}$ and $I \subseteq W^i$ with $|I| = m-1$, define the polytope
\[
\textstyle
P^{(i,I)} := \{ \sum_{w\in I} \lambda_w w : \lambda_w \in [0,1] ~\forall ~ w \in I \}.
\]
By assumption, $w \in \Lambda$ for each $w \in I$, so the vertices of $P^{(i,I)}$ are in $\Lambda$.
Ehrhart theory then implies that there is a polynomial $\pi^{(i,I)}(t)$ of degree $m-1$ such that
\[
\textstyle
\pi^{(i,I)}(t) = | tP^{(i,I)}  \cap \Lambda | =  |\{ \sum_{w \in I} \lambda_w w : \lambda_w \in [0,t] ~ \forall ~ w \in I \} \cap \Lambda|
\]
for each $t \in \Z_{\ge 1}$.
The leading coefficient of $\pi^{(i,I)}$ is the $(m-1)$ dimensional volume of $P^{(i,I)}$, which is positive, see~\cite[Chapter VIII]{barv}.
Similarly, for the polytope
\[
\textstyle
P^{i} := \{ \sum_{w \in W^i} \lambda_w w : \lambda_w \in [0,1] ~\forall ~ w \in W^i \}
\]
there exists a polynomial $\pi^i(t)$ of degree $m$ with positive leading coefficient such that for each $t \in \Z_{\ge 1}$
\[
\textstyle
\pi^i(t) = |tP^{i} \cap \Lambda | =  | \{ \sum_{w \in W^i} \lambda_w w : \lambda_w \in [0,t]  ~\forall~ w \in W^i\} \cap \Lambda |.
\]

Define
\[
\textstyle
d := \max\{ \| r + \sum_{w \in I} w \| _{\infty} : i \in \{1, \dotsc, s\}, ~ r \in L^i,~ I \subseteq W^i \text{ with } |I| \le m-1\}.
\]
We show that the values in~\eqref{eqLimit0} go to zero as $t \to \infty$ by bounding the fraction
\[
\frac{|\{-td, ..., td\}^m \cap \bigcup_{i=1}^s (\Lambda \cap [K^i \setminus (K^i+z^i)])|}
{|\{-td, ..., td\}^m \cap  \bigcup_{i=1}^s (\Lambda \cap K^i)|}
\]
for each $t \in \Z_{\ge 0}$.
By the definition of $d$, $t P^{i} \subseteq \{-td, \dotsc, td\}^m \cap K^i$ for every $i \in \{1, \dotsc, s\}$.
So for each $i \in \{1, \dotsc, s\}$, say $i =1$, it follows that
\[
\pi^1(t) =   |tP^1 \cap \Lambda | \le |\{-td, \dotsc, td\}^m \cap \Lambda \cap K^1| \le \bigg|\{-td, ..., td\}^m \cap  \bigcup_{i=1}^s (\Lambda \cap K^i)\bigg|.
\]
Hence, 
\begin{equation}\label{eqBound1}
\frac{1}{|\{-td, ..., td\}^m \cap  \bigcup_{i=1}^s (\Lambda \cap K^i)|}
\le 
\frac{1}{\pi^1(t)}.
\end{equation}

If $i \in \{1, \dotsc, s\}$ and $y \in \{-td, \dotsc, td\}^m \cap \Lambda \cap [K^i \setminus (K^i+z^i)]$, then, by~\eqref{eqBadSet}, $y = r + \sum_{w \in I} \lambda_w w$ for $r \in L^i$, $I \subseteq W^i$ with $|I| = m-1$, and $\lambda_w \in \R_{\ge 0}$ for each $w \in I$.  
This implies that
\[
\textstyle
\|\sum_{w \in I} \lambda_w w\|_{\infty} = \| y - r \|_{\infty}  \le \|y\|_{\infty} + \|r\|_{\infty} \le td + d = (t+1)d.
\]
Hence,
\[
\{-td, \dotsc, td\}^m \cap \Lambda \cap [K^i \setminus (K^i+z^i) ]\subseteq 
\bigcup_{r \in L^i} ~
\bigcup_{\substack {I \subseteq W^i\\ |I| = m-1}} r + (t+1)d P^{(i,I)}.
\]
If $r \in L^i$, then by the definition of $L^i$, $r \in \Lambda$.
This implies that the number of $\Lambda$ points in $r + (t+1)d P^{(i,I)}$ is equal to $\pi^{(i,I)}((t+1)d)$.
So,
\begin{align}
 |\{-td, \dotsc, td\}^m \cap \bigcup_{i=1}^s (\Lambda \cap [K^i \setminus (K^i+z^i)])|
 \le 
 \sum_{i=1}^s ~ 
\sum_{r \in L^i} ~
\sum_{\substack {I \subseteq W^i\\ |I| = m-1}} \pi^{(i,I)}((t+1)d).\label{eqBound2}
\end{align}

The polynomial on the right-hand side of~\eqref{eqBound2}, call it $\psi(t+1)$, is of degree $m-1$ and has a positive leading coefficient.
Also, by~\eqref{eqBound1} and~\eqref{eqBound2}, 
\[
\frac{|\{-td, ..., td\}^m \cap \bigcup_{i=1}^s (\Lambda \cap [K^i \setminus (K^i+z^i)])|}
{|\{-td, ..., td\}^m \cap \bigcup_{i=1}^s (\Lambda \cap K^i)|} 
 \le \frac{\psi(t+1)}{\pi^1(t)}.
\]
Recall that $\pi^1$ is of degree $m$, $\psi$ is of degree $m-1$, and $\psi$ and $\pi^1$ have positive leading coefficients.
Moreover, the limit as $t \to \infty$ is the same as $td \to \infty$.
Hence, 
\begin{align*}
 &\lim_{t \to \infty} \frac{|\{-t, ..., t\}^m \cap \bigcup_{i=1}^s (\Lambda \cap [K^i \setminus (K^i+z^i)])|}
{|\{-t, ..., t\}^m \cap \bigcup_{i=1}^s (\Lambda \cap K^i)|}  \\
= & 
\lim_{t \to \infty}  \frac{|\{-td, ..., td\}^m \cap \bigcup_{i=1}^s (\Lambda \cap [K^i \setminus (K^i+z^i)])|}
{|\{-td, ..., td\}^m \cap \bigcup_{i=1}^s (\Lambda \cap K^i)|} 
=  \lim_{t \to \infty} \frac{\psi(t+1)}{\pi^1(t)} = 0. \hspace{.25 in} \qed
\end{align*}
\end{proof}

\section{Proof of Theorem~\ref{thmMainProb}}\label{secMainProb}

The assumption $\cone(A) = \cone(W)$ indicates that we can write $\cone(A)$ as
\begin{equation}\label{eqCoverA}
\textstyle \cone(A) = \bigcup_{i=1}^s \cone(W^i),
\end{equation}
where $s \in \Z_{\ge 1}$ and $W^1, \dotsc, W^s \subseteq W$ are linearly independent sets; see~\eqref{eqCover}.
Also, $A$ has full row rank, so we assume that $W^1, \dotsc, W^s$ each contain $m$ linearly independent vectors.
For $i \in \{1, \dotsc, s\}$, let $K^i := \cone(W^i)$.

First, we prove $\sigma^{\asy}(A) \le m + \phi^{\max}(A)$.
In order to do this, we find a lattice $\Lambda$ and points $z^1 \in K^1, \dotsc, z^s \in K^s$ such that 
\[
\sigma(A,b) \le m + \phi^{\max}(W) \quad \forall ~ b \in (\Lambda \cap (K^1 + z^1)) \cup \dotsc \cup (\Lambda \cap (K^s + z^s))
\]
and $\Lambda$ contains every $ b \in \Z^m $ such that $ P(A,b)  \neq \emptyset $.
With these values, we will be able to apply Lemma~\ref{lemAsymptoticGeneral} to prove the desired result.

Fix $i \in \{1, \dotsc, s\}$ and set
\( 
\phi^i := \phi(W^i).
\)
Let $G_{W^i}(\Z^m)$ be the group defined in Section~\ref{secGroup}.
In view of Lemma~\ref{lemGroupDecomp}, there exist $w^1, \dotsc, w^{t} \in A$ with $t \le \phi^i$ and 
\[
G_{W^i}(\{w^1, \dotsc, w^{t}\}) = G_{W^i}(A).
\]
We emphasize that the choice of $w^1, \dotsc, w^{t}$ depends on $W^i$. 
Define the lattice 
\[
\Lambda^i  := \bigg\{\sum_{h \in G_{W^i}(A)} k_h h + \sum_{w \in W^i} p_w w : k_h \in \Z ~ \forall ~ h \in G_{W^i}(A),~ p_w \in \Z ~\forall ~ w \in W^i \bigg\}.
\] 
In Lemma~\ref{lemmaAllSame}, we show that $\Lambda^i$ does not depend on $i$.
Lemma~\ref{lemSubgroup} implies that $\Lambda^i \supseteq \{g \in G_{W^i}(\Z^m) : \exists ~ b \in \Z^m \text{ such that }\rho_{W^i}(b) = g \text{ and }P(A,b) \neq \emptyset\}$.
Thus,
\begin{equation}\label{eqInfeasibleCond}
\text{if } b \not \in \Lambda^i \text{ (equivalently, if } \rho_{W^i}(b) \not \in G_{W^i}(A))\text{, then }  P(A,b) = \emptyset.
\end{equation}

\begin{lemma}\label{lemmaSuffCond1} 
There exists $z^i \in \Lambda^i \cap K^i $ that satisfies the following: for every $ b \in (K^i + z^i)\cap \Z^m $, either $ b \not \in \Lambda^i $ (so $P(A,b) = \emptyset$) by~\eqref{eqInfeasibleCond}) or $\sigma(A,b) \le m + \phi^i$.
The vector $z^i$ satisfies $z^i = \sum_{w \in W^i}k_w w$, where $k_w \in \Z_{\ge 0}$ for each $w \in W^i$.
\end{lemma}
\begin{cpl}
For each $g \in G_{W^i}(A) = G_{W^i}(\{w^1, \dotsc, w^{t}\})$, there exists $x^g \in \Z^m$ such that
\begin{equation}\label{eqResidueReps}
\textstyle
x^g - g = \sum_{w \in W^i} \tau_w w \quad \text{and} \quad x^g = \sum_{w \in W^i} q_w w + \sum_{\ell=1}^{t}  p_{\ell} w^{\ell},
\end{equation}
where $\tau_w \in \Z$ and $q_w \in \Z_{\ge 0}$ for each $w \in W^i$ and $p_1, \dotsc, p_{t} \in \Z_{\ge 0}$. 
By Lemma~\ref{lemOverlappingConesMany}, there exists 
\(
z^i \in \Lambda^i \cap K^i 
\)
such that $\rho_{W^i}(z^i) = 0$ and
\(
K^i + z^i \subseteq K^i  \cap \bigcap_{g \in G_{W^i}(A)} (K^i + x^g).
\)
Let $b \in (K^i + z^i) \cap \Z^m$ such that $P(A,b) \neq \emptyset$.
By~\eqref{eqInfeasibleCond}, there is a $g \in G_{W^i}(A)$ such that $\rho_{W^i}(b) = g$.
So, by~\eqref{eqResidueReps},
\begin{equation}\label{eqFinalRep}
\textstyle
\begin{array}{rcl}
b 
= g + \sum_{w \in W^i} \bar{\tau}_{w} w 
&=& x^g + \sum_{w \in W^i} ( \bar{\tau}_{w} -  \tau_w) w\\[.15 cm]
&= &\sum_{\ell=1}^t  p_{\ell} w^{\ell} + \sum_{w \in W^i} (q_{w} + \bar{\tau}_w -   \tau_w ) w,
\end{array}
\end{equation}
where $\bar{\tau}_w \in \Z$ for each $w \in W^i$.
Note $\bar{\tau}_{w} -  \tau_w \in \Z_{\ge 0}$ for each $w \in W^i$ because $b \in K^i + z^i \subseteq K^i + x^g$.
Thus, $P(A,b) \neq \emptyset$ and $\sigma(A,b) \le |W^i| + t \le m + \phi^i$.
\end{cpl}

\begin{lemma}\label{lemmaAllSame} 
For every pair $i,j \in \{1, \dotsc, s\}$, the lattices $\Lambda^i$ and $\Lambda^j$ are equal.
\end{lemma}
\begin{cpl}
It is enough to show that $\Lambda^1 \subseteq \Lambda^2$.
Let $x \in \Lambda^1$.
By Lemmata~\ref{lemOverlappingConesMany} and~\ref{lemmaSuffCond1}, there is a point $y \in (K^1 + z^1) \cap \Lambda^1$ such that $\rho_{W^1}(y) = \rho_{W^1}(x)$. 
Also, by Lemma~\ref{lemmaSuffCond1}, $P(A,y) \neq \emptyset$.
Hence, by~\eqref{eqInfeasibleCond}, $y \in \Lambda^2$.
Similarly, $w \in \Lambda^2$ for each $w \in W^1$. 
These inclusions along with $\rho_{W^1}(y) = \rho_{W^1}(x)$ imply $x \in \Lambda^2$.
\end{cpl}

Set $\Lambda := \Lambda^1 = \dotsc = \Lambda^s$.
Lemma~\ref{lemmaSuffCond1} implies that
\[
\textstyle 
\bigcup_{i=1}^s (\Lambda \cap (K^i+z^i)) \subseteq \{b \in \Z^m : \sigma(A,b) \le m +\phi^{\max}\}.
\]
By~\eqref{eqCoverA} and~\eqref{eqInfeasibleCond}, it follows that
\[
\textstyle
\{b \in \Z^m : P(A,b) \neq \emptyset\} 
\subseteq 
\cone(A) \cap \Lambda
= 
\bigcup_{i=1}^s \Lambda \cap K^i
\]
Hence, for each $t \in \Z_{\ge 1}$, it follows that
\begin{align}\label{eqFinalLimit} 
& ~ \frac
{|\{b \in \{-t, ..., t\}^m : \sigma(A,b) \le m +\phi^{\max}\}|}
{|\{b \in \{-t, ..., t\}^m : P(A,b) \neq \emptyset\}|} \nonumber\\[.1 cm]
\ge & ~
\frac
{|\{b \in \{-t, ..., t\}^m \cap \left( \bigcup_{i=1}^s \Lambda \cap (K^i+z^i) \right) \}|}
{|\{b \in \{-t, ..., t\}^m  \cap \left( \bigcup_{i=1}^s \Lambda \cap K^i \right)\}|}.
\end{align}

By Lemma~\ref{lemAsymptoticGeneral}, it follows that $\sigma^{\asy}(A) \le m +\phi^{\max}(W)$.
Also, the inequality $\phi^i \le \log_2(|\det(W^i)|)$ for each $i \in \{1, \dotsc, s\}$ implies $\phi^{\max}(W) \le \log_2(\delta^{\max}(W))$ and $\sigma^{\asy}(A) \le m + \phi^{\max}(W) \le m + \log_2(\delta^{\max}(W))$.
Consider the inequality $\sigma^{\asy}(A) \le 2m +\phi^{\min}(W)$.
Without loss of generality, $\phi^1 \le \dotsc \le \phi^s$.
Let $z^1 \in K^1 \cap \Lambda$ be given from Lemma~\ref{lemmaSuffCond1}.
Let $i \in \{2, \dotsc, s\}$.
Using Lemma~\ref{lemOverlappingConesMany} and the fact that $K^1 + z^1$ is $m$-dimensional, the representative set $\{x^g : g \in G_{W^i}(A)\}$ from~\eqref{eqResidueReps} can be chosen in $K^1 + z^1$.
Let $b \in K^i+z^i$.
By~\eqref{eqFinalRep}, there exists a $g \in G_{W^i}(A)$ such that
\[
\textstyle 
b= x^g + \sum_{w \in W^i} ( \bar{\tau}_{w} -   \tau_w ) w,
\]
where $\bar{\tau}_{w} -   \tau_w \in \Z_{\ge 0}$ for each $w \in W^i$.
The point $x^g$ is in $K^1 + z^1$, so $P(A, x^g) \neq \emptyset$ and there are $w^1, \dotsc, w^{m+ \phi_1} \in A$ such that
\(
x^g = \sum_{i=1}^{m+ \phi_1} q_\ell w^{\ell},
\)
where  $q_1, \dotsc, q_{m+\phi_1} \in \Z_{\ge 0}$.
So,
\[ 
\textstyle
b= \sum_{\ell=1}^{m+ \phi_1} q_\ell w^{\ell} + \sum_{w \in W^i} ( \bar{\tau}_{w} -   k_w ) w.
\]
Thus, $P(A,b) \neq \emptyset$ and $\sigma(A,b) \le 2m + \phi^1 = 2m +\phi^{\min}(W)$.
Hence, $\sigma^{\asy}(A) \le 2m +\phi^{\min}(W)$.

Finally, assume $\log_2(\delta^{\min}(W)) = \log_2(W^2)$.
Observe that $\phi(W^2) \le \log_2(W^2)$, so $\sigma^{\asy}(A) \le 2m + \phi^{\min}(W) \le 2m + \log_2(\delta^{\min}(W))$.
 \qed

\subsection*{Acknowledgements}

The authors would like to thank anonymous referees for helping to improve the presentation of the paper and Marie Putscher for identifying a typo in the proof of Lemma 4. 


\bibliographystyle{splncs04}
\bibliography{references}


\appendix

\section{Proof of Theorem~\ref{thmMainLower}}\label{secMainLower}

We construct both matrices $A$ and $B$ using a submatrix $\tilde{A}$, which we construct first.
Let $d \in \Z_{\ge 1}$ and $p_1 < \dotsc < p_d$ be prime. 
For $i \in \{1, \dotsc, d\}$, define 
\(
q_i := \prod_{j=1, j \neq i}^d p_i
\)
and $\delta := \prod_{j=1}^d p_i$.
Define the matrix
\(
\tilde{A} := \begin{bmatrix} q_1, & \dotsc & q_d, & - \delta \end{bmatrix}.
\)
The matrix $\tilde{A}$ has $d+1$ columns, so $\sigma^{\asy}(\tilde{A}) \le 1+d$.
The matrix $\tilde{A}$ is similar to the example in~\cite[Theorem 2]{ADEOW2018} and the theory of so-called \emph{primorials}. 
We claim 
\begin{equation}\label{eqLowerBound}
\text{if } b \in \Z_{< 0} \text{ and } b \equiv 1 \text{ mod } \delta, \text{ then } P(\tilde{A},b) \neq \emptyset \text{ and } \sigma(\tilde{A},b) = 1 + d.
\end{equation}
Note that $\gcd(q_1, \dotsc, q_d) = 1$.
The \emph{Frobenius number} of $\{q_1, \dotsc, q_d\}$ is the largest integer that cannot be written as a positive integer linear combination of $q_1, \dotsc,$ and $q_d$.
Hence, if we choose $\bar{b} \in \Z_{\ge 1}$ to be the Frobenius number of $\{q_1, \dotsc, q_d\}$, then $b \ge \bar{b}+1$ implies $P(\tilde{A},b) \neq \emptyset$.
If $b \equiv 1 \text{ mod } \delta$, then $b$ is not divisible by $p_i$ for any $i \in \{1, \dotsc, d\}$.
Thus, if $b \ge \bar{b}+1$ and $b \equiv 1 \text{ mod } \delta$, then $\sigma(\tilde{A}, b) = d$. 
Finally, observe that if $b < 0$, then $b + k\delta  >\bar{b}$ for large enough $k \in \Z_{\ge 1}$.
The only negative column of $\tilde{A}$ is $-\delta$, so $\sigma(\tilde{A},b) = 1+d$.
This proves~\eqref{eqLowerBound}.

Now we define the matrix $A$.
Let $m \in \Z_{\ge 1}$ and define
\[
A := 
\begin{bmatrix}
I^{m-1}  ~ &  0^{(m-1)\times (d+1)} \\  
 0^{1 \times (m-1)} & \tilde{A} 
\end{bmatrix} \in \Z^{m \times (m+d)},
\]
where $I^k \in \Z^{k \times k}$ is the identity matrix and $0^{k \times s} \in \Z^{k \times s}$ is the all zero matrix for $k,s\in \Z_{\ge 1}$.
Note that $\phi^{\max}(A) = d$.
If $b \in \Z^{m-1}_{> 0} \times \Z_{<0}$ is such that the last component is equivalent to $1 \text{ mod } \delta$, then $\sigma(A,b) = m+d$ by the arguments above.
Now, the set of $b\in \Z^m$ such that $P(A,b) \neq \emptyset$ is contained in $\Z^{m-1}_{\ge 0} \times \Z$.
So, for every $t \in \Z_{\ge 1}$, the set of feasible solutions in $\{-t\delta, \dotsc, t\delta\}^m$ contains $t(t\delta-1)^{m-1}$ points $b$ such that $\sigma(A,b) = m+d$.
Moreover, if $t \in \Z_{\ge \bar{b}}$, then $P(A,b) \neq \emptyset$ for every $b \in \{0, \dotsc, t\delta\}^{m-1}\times \{-t\delta, \dotsc, t\delta\}$.
Therefore,
\[
\begin{array}{rclcl}
 &\displaystyle \lim_{t \to \infty} \frac
{|\{b \in \{-t, ..., t\} : \sigma(A,b) \le (m-1) + d\}|}
{|\{b \in \{-t, ..., t\} : P(A,b) \neq \emptyset \}|} \\[.5 cm]
= &\displaystyle \lim_{t \to\infty}
\frac{|\{b \in \{-t\delta, ..., t\delta\} : \sigma(A,b) \le (m-1) + d\}|}
{|\{b \in \{-t\delta, ..., t\delta\} : P(A,b) \neq \emptyset \}|} \\[.5 cm]
\le & \displaystyle
\lim_{t \to\infty}
\frac{(2t\delta+1)(t\delta+1)^{m-1} - t(t\delta+1)^{m-1}}
{(2t\delta+1)(t\delta+1)^{m-1}} & <& 1.
\end{array}
\]
Using this and the fact that $A$ has $m + d$ columns, we have $\sigma^{\asy}(A) = m+d$.

Now we define the matrix $B$.
Let $A \in \Z^{m\times (m+d)}$ be as above.
Let $e^{1 \times (m+1)} \in \Z^{1 \times (m+1)}$ be the all ones matrix and $U \in \Z^{m\times (m+1)}$.
Assume
\[
\bigg|
\det\bigg(\begin{bmatrix}
U \\
e^{1 \times (m+1)}
\end{bmatrix}
\bigg)\bigg| =1
\]
and set
\[
B := 
\begin{bmatrix}
U  ~ & A \\
e^{1 \times (m+1)} ~ &  0^{1 \times (m+d)} 
\end{bmatrix} \in \Z^{(m+1) \times (2m+1+d)}.
\]
Note that $\phi^{\min}(B) = 0$, so Theorem~\ref{thmMainProb} \emph{(ii)} implies that $\sigma^{\asy}(B) \le 2m+2$.
Let $b \in \Z^{m}\times\{0\} $ be such that $P(B,b) \neq \emptyset$.
If $z \in P(B,b)$, then the first $m+1$ components of $z$ are zero.
So, similarly to above, there are $b \in \Z^{m+1} $ such that $\sigma(B,b) = m+d$.
Hence, $\sigma^{asy}(B) \le 2m+2 < m+d = \sigma(B)$.
 \qed

%
%

\section{Proof of Lemma~\ref{lemOverlappingConesMany}}\label{appProoflemOverlappingConesMany}

Assume that $t = 2$.
Let $x := x^1$ and $y := x^2$.
First, we show that $K \cap (K + x) \cap (K+y) \neq \emptyset$.
Since $v^1, \dotsc, v^m$ are linearly independent, $K$ is a full-dimensional simplicial cone.
Hence, there exist linearly independent vectors $a^1, \dots, a^m \in \R^m $ such that $K = \{ w \in \R^m : (a^i)^\intercal w \le 0 ~ \forall ~ i \in \{1, \dots, m\}\}$ and linearly independent vectors $r^1, \dots, r^m \in K$ such that $(a^i)^\intercal r^i < 0$ for each $i \in \{1, \dots, m\}$.

There is a set $ J \subseteq \{1, \cdots, m\}$ such that $(a^j)^\intercal (x-y) > 0$ for each $j \in J$ and $(a^j)^\intercal (x-y) \le 0$ for each $j \in \{1, \dots, m\}\setminus J$.
For $ j \in \{1, \dots, m\}$, set 
\[
\lambda_j := \begin{cases}
\max\left\{ 0, -\frac{(a^j)^\intercal x }{(a^j)^\intercal r^j}\right\}, & \text{if } j \in \{1, \dotsc, m\} \setminus J\\\\
\max\left\{-\frac{(a^j)^\intercal (x-y) }{ (a^j)^\intercal r^j}, -\frac{(a^j)^\intercal x }{(a^j)^\intercal r^j}\right\}, & \text{if } j \in J.
\end{cases}
\]
Note that $\lambda_1, \dots, \lambda_m \in \R_{\ge 0}$, so 
\(
x +\sum_{j=1}^m \lambda_j r^j \in K + x.
\)
For each $i \in \{1, \dots, m\}$, it follows that
\[
(a^i)^\intercal \bigg(x +\sum_{j=1}^m \lambda_j r^j - y \bigg) \le (a^i)^\intercal (x-y) + \lambda_i (a^i)^\intercal r^i \le 0.
\]
So, $x +\sum_{j=1}^m \lambda_j r^j -y \in K$ and $x +\sum_{j=1}^m \lambda_j r^j \in K + y$.
Finally, for each $i \in \{1, \dots, m\}$, it follows that
\[
(a^i)^\intercal \bigg(x +\sum_{j=1}^m \lambda_j r^j \bigg) \le (a^i)^\intercal x + \lambda_i (a^i)^\intercal r^i \le 0.
\]
Hence, $x +\sum_{j=1}^m \lambda_j r^j \in K$ and $ K \cap (K+x) \cap (K+y) \neq \emptyset$.

Let $w \in K \cap (K+x) \cap (K+y) $. 
Then $K+ w \subseteq K \cap (K+x) \cap (K+y)$.
Because $K$ is full-dimensional, there exists a point $z \in (K+w) \cap \Z^m$ such that $z = \sum_{i=1}^m k_i v^i$ for $k_i \in \Z_{\ge 0}$.
Note that 
\(
z \in K + w \subseteq K
\)
and
\[
K + z \subseteq K + w \subseteq K + \left( K \cap (K+x) \cap (K+y)\right)  \subseteq K \cap (K+x) \cap (K+y).
\]

For $t \ge 3$, the result follows by induction.
\qed

\end{document}